\documentclass[]{article}

\usepackage{amssymb}
\usepackage{amsmath}
\usepackage{algpseudocode}
\usepackage{url,here}
\usepackage[backref,colorlinks=true]{hyperref}
\usepackage{multirow}
\usepackage{algorithm}
\allowdisplaybreaks

\begin{document}

\newenvironment {proof}{{\noindent\bf Proof.}}{\hfill $\Box$ \medskip}

\newtheorem{theorem}{Theorem}[section]
\newtheorem{lemma}[theorem]{Lemma}
\newtheorem{condition}[theorem]{Condition}
\newtheorem{proposition}[theorem]{Proposition}
\newtheorem{remark}[theorem]{Remark}
\newtheorem{definition}[theorem]{Definition}
\newtheorem{hypothesis}[theorem]{Hypothesis}
\newtheorem{corollary}[theorem]{Corollary}
\newtheorem{example}[theorem]{Example}
\newtheorem{descript}[theorem]{Description}
\newtheorem{assumption}[theorem]{Assumption}

\newcommand{\ba}{\begin{align}}
\newcommand{\ea}{\end{align}}

\def\P{\mathbb{P}}
\def\R{\mathbb{R}}
\def\E{\mathbb{E}}
\def\N{\mathbb{N}}
\def\Z{\mathbb{Z}}

\renewcommand {\theequation}{\arabic{section}.\arabic{equation}}
\def \non{{\nonumber}}
\def \hat{\widehat}
\def \tilde{\widetilde}
\def \bar{\overline}

\def\ind{{\mathchoice {\rm 1\mskip-4mu l} {\rm 1\mskip-4mu l}
{\rm 1\mskip-4.5mu l} {\rm 1\mskip-5mu l}}}

\title{\Large\ { \bf Determining the long-term behavior of cell populations: A new procedure for detecting ergodicity in large stochastic reaction networks}}

\author{Ankit Gupta and Mustafa Khammash \\
Department of Biosystems Science and Engineering \\ ETH Zurich \\  Mattenstrasse 26 \\ 4058 Basel, Switzerland. 
}
\date{\today}

\maketitle
\begin{abstract}
A reaction network consists of a finite number of species, which interact through predefined reaction channels. Traditionally such networks were modeled deterministically, but it is now well-established that when reactant copy numbers are small, the random timing of the reactions create internal \emph{noise} that can significantly affect the macroscopic properties of the system. To understand the role of noise and quantify its effects, stochastic models are necessary. In the stochastic setting, the population is described by a probability distribution, which evolves according to a set of ordinary differential equations known as the Chemical Master Equation (CME). This set is infinite in most cases making the CME practically unsolvable. In many applications, it is important to determine if the solution of a CME has a globally attracting fixed point. This property is called \emph{ergodicity} and its presence leads to several important insights about the underlying dynamics. The goal of this paper is to present a simple procedure to verify ergodicity in stochastic reaction networks. We provide a set of simple linear-algebraic conditions which are sufficient for the network to be ergodic. In particular, our main condition can be cast as a \emph{Linear Feasibility Problem} (LFP) which is essentially the problem of determining the existence of a vector satisfying certain linear constraints. The inherent scalability of LFPs make our approach efficient, even for very large networks. We illustrate our procedure through an example from systems biology.
\end{abstract}

\medskip

\setcounter{equation}{0}

\section{Introduction}

Reaction networks represent a modeling paradigm that is used in many biological disciplines, such as, systems biology, epidemiology, pharmacology and ecology. Such networks were traditionally studied by expressing the dynamics as a set of ordinary differential equations. However these deterministic formulations become inaccurate when the reactant copy numbers are \emph{small}. In this case, the  discrete nature of the interactions makes the dynamics inherently \emph{noisy} and this noise can have a significant impact on the macroscopic properties of the system (see \cite{Elowitz}). To account for this noise and study its effects, a stochastic formulation of the dynamics is necessary. The most common approach is to model the dynamics as a continuous-time Markov process whose states denote the current population size of the constituent species. Many recent articles use such stochastic models to understand the role of noise in various biological phenomena.

Even though stochastic models have become very popular, the tools for analyzing them are still lacking.
Most papers that use such models have to simulate several trajectories (using the \emph{Stochastic Simulation Algorithm} by Gillespie \cite{GP}, for example) in order to determine the relevant characteristics of the system. Simulation of trajectories can be computationally demanding, and since one can only simulate a finite number of trajectories for a finite amount of time, properties like long-term behaviour cannot be satisfactorily studied through such simulations. Our goal in this paper is to overcome this problem and provide a direct way to examine the long-term behaviour for the stochastic model, without relying on simulations. Specifically we check if the underlying Markov process is \emph{ergodic}, which is analogous to having a globally attracting fixed point in the deterministic setting.  An ergodic process has a unique stationary distribution, and in the long-run, the proportion of time spent by its trajectories in any set is equal to the stationary probability of that set (see \eqref{resultLLN}). Hence information about the whole population at stationarity can be obtained by observing just one trajectory for a long time. Such an insight can be used to leverage different experimental techniques such as \emph{flow-cytometry} and \emph{time-lapse microscopy}, for biological applications. Ergodicity also implies that certain moments of the underlying Markov process converge to their steady-state values with time (see \eqref{ergodicconvergenceoff}). This can be used to design biological controllers that steer the moments to specific steady state values.

The canonical example of an ergodic reaction network is the simple birth-death model in which a single chemical species $\mathbf{S}$ undergoes the following two reactions:
\begin{align*}
\emptyset \stackrel{\theta_1 }{\rightarrow} \mathbf{S} \stackrel{\theta_2 }{\rightarrow} \emptyset,
\end{align*}
where $\theta_1,\theta_2 >0$. For this network, the reaction dynamics is given by a Markov process $(X(t))_{t \geq 0}$ with state space $\N_0 = \{0,1,2,\dots\}$. At any time $t$, $X(t)$ is the number of molecules of species $\mathbf{S}$. If $X(t) = n$, then the next reaction occurs at time $(t+\tau)$, where $\tau$ is an exponentially distributed random variable with rate $(\theta_1+ \theta_2 n)$. At time $(t+\tau)$ the state jumps by $\pm 1$ with probabilities $p_{\pm}(n)$ given by 
\begin{align*}
 p_{+}(n)= \left( \frac{\theta_1}{ \theta_1 + \theta_2 n } \right)  \textnormal{ and }    p_{-}(n)= \left( \frac{\theta_2 n}{ \theta_1 + \theta_2 n } \right). 
\end{align*}
From these probabilities, two observations can be made. Firstly, the state space $\N_0$ is \emph{irreducible}, which means that there is a positive probability for reaching any state in $\N_0$ from any other state in $\N_0$, in a finite time. Secondly, if the current state $X(t) = n$ is large, the stochastic dynamics experiences a \emph{negative drift} in the sense that the next jump state is more likely to be below $n$ than above $n$. Establishing \emph{irreducibility} of the state space and checking the \emph{negative drift} conditions will be the two main steps in proving ergodicity for a general reaction network.

The approach we present in this paper, relies on some known results on stochastic processes and it involves checking simple linear-algebraic conditions. In particular, we would need to solve \emph{Linear Feasibility Problems} (LFPs) of the form :
\begin{equation}
\label{lfp:form}
\mathcal{F} = \left\{  v \in \R^n :  A v \leq b \textnormal{ and } A_{\textnormal{eq} } v = b_{\textnormal{eq} } \right\},
\end{equation}
for certain matrix-vector pairs $(A,b)$ and $(A_{\textnormal{eq} } , b_{ \textnormal{eq} } )$. We say that the LFP corresponding to set $\mathcal{F} $ has a solution if this set is non-empty. Many methods are available to efficiently solve LFPs in very high dimensions. Therefore our approach can be easily applied to very large networks.

This paper is organized as follows. In Section \ref{sec:prelim} we provide some mathematical background. Our main results are presented in Section \ref{sec:mainresults} and in Section \ref{sec:example} we illustrate our approach through an example.

{\bf Notation :} We now introduce some notation that will be used in the paper. Let $\R$, $\R_+$, $\Z$, $\N$ and $\N_{0}$ denote the sets of all reals, nonnegative reals, integers, positive integers and nonnegative integers respectively. For $v,w \in \R^n$ we say $v <w$ or $v \leq w$ if the corresponding inequality holds component-wise. The vectors of all zeros and all ones in $\R^n$ are denoted by $\bar{0}_n$ and $\bar{1}_n$ respectively. For any $v=(v_1,\dots,v_n) \in \R^n$ we define its support as $\textnormal{supp}(v) = \{ i =1,\dots,n  : v_i \neq 0 \}$. Let $M$ be a $m \times n$ matrix with real entries. We denote its rank by $\textnormal{Rank}(M)$. If $C_1,\dots,C_n$ are the columns of $M$ then for any $A \subset \R$, the set $\textnormal{Colspan}_A(M)$ stands for
\begin{align*}
\left\{x \in \R^m :  x = \sum_{i=1}^n a_i C_i \textnormal{ for some } a_1,\dots,a_n \in A\right\}.
\end{align*}
For any positive integer $n$, where $I_n$ is the $n \times n$ identity matrix.
While multiplying a matrix with a vector we always regard the vector as a column vector. 
\section{Preliminaries} \label{sec:prelim}

We start by formally defining the stochastic model of a reaction network. Consider a system containing molecules that belong to one of $d$ species $\mathbf{S}_1,\dots,\mathbf{S}_d$. We assume that the system is well-stirred and hence its state at any time can be described by a vector in 
$\mathbb{N}_{0}^d$, whose $i$-th component is the number of molecules of the $i$-th specie. The species interact through $K$ predefined reaction channels. For any $k = 1,\dots,K$, the $k$-th reaction has the form
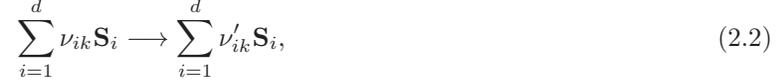
\begin{align}
\label{formofthereaction}
\sum_{i=1}^d \nu_{ik} \mathbf{S}_i \longrightarrow \sum_{i=1}^d \nu'_{ik} \mathbf{S}_i,
\end{align} 
where $\nu_{ik}$ ($\nu'_{ik} $) denotes the number of molecules of species $\mathbf{S}_i$ that are consumed (produced) by reaction $k$.  Let $\nu_k$ and $\nu'_k$ be vectors in $\N^d_0$, given by  $\nu_{k} = (\nu_{1k},\dots, \nu_{dk} )$ and $\nu'_{k} = (\nu'_{1k},\dots, \nu'_{dk} ) $.
When the state of the system is $x = (x_1,\dots,x_d)$, the $k$-th reaction fires after a random time which is exponentially distributed with rate $\lambda_k(x)$ and it displaces the state by $(\nu'_{k} - \nu_{k})$. The functions $\lambda_1,\dots,\lambda_K$ are called the \emph{propensity} functions for the reaction network. We assume \emph{mass action kinetics} and hence each $\lambda_k$ is given by
\begin{align}
\label{massactionkinetics}
\lambda_k(x_1,\dots,x_d) = \theta_k \prod_{i =1}^d \frac{ x_i(x_i-1)\dots (x_i - \nu_{ik} +1 ) }{  \nu_{ik} ! },
\end{align}
where $\theta_k > 0$ is the rate constant for the $k$-th reaction. 

The property of ergodicity depends crucially on the choice of the state space $\mathcal{S}$ for the reaction dynamics. We will later discuss how it can be chosen appropriately. For now, let $\mathcal{S}$ be a non-empty subset of $\N^d_0$ which satisfies the following property : if $y \in \mathcal{S}$ and $\lambda_k(y) > 0$ for some $k=1,\dots,K$, then $y + \zeta_k \in \mathcal{S}$. This property ensures that if the reaction dynamics starts in $\mathcal{S}$ then it stays in $\mathcal{S}$ forever. Let $(X(t))_{ t \geq 0}$ be the Markov process representing the stochastic reaction dynamics with some initial state $X(0)$ in $\mathcal{S}$. 
For any $x,y \in \mathcal{S}$ let 
\begin{align}
\label{transition_probabilities}
p_x(t,y) = \P\left( X(t) = y \vert X(0) =x \right).
\end{align}
Hence $p_x(t,y) $ is the probability that the reaction dynamics starting at $x$ will be in state $y$ at time $t$. Defining $p_x(t,A) = \sum_{y \in A} p_x(t,y)$ for any $A \subset \mathcal{S}$, we can view $p_x(t)$ as a probability distribution over $\mathcal{S}$. The dynamics of $p_x(t)$ is given by the Chemical Master Equation (CME) which has the following form. For each $y \in \mathcal{S}$
\begin{align*}
\frac{ d p_{x}(t,y) }{dt} = & \sum_{k=1}^K  ( p_{x}(t, y -\zeta_k) \lambda_k(y - \zeta_k) -p_{x}(t, y ) \lambda_k(y) ) 
\end{align*} 
where $\zeta_k = \nu'_k - \nu_k$. Observe that this system consists of as many equations as the number of elements in $\mathcal{S}$, which is typically infinite and hence solving this system is nearly impossible.

Note that the CME essentially describes a dynamical system over the space of probability measures on $\mathcal{S}$. We are interested in knowing if this dynamical system has a globally attracting fixed point. Specifically, we would like to determine if there exists a probability distribution $\pi$ over $\mathcal{S}$ such that
\begin{align}
\label{ergodicconvergence}
\lim_{t \to \infty} \sup_{A \subset \mathcal{S}} \left| p_{x}(t,A) - \pi(A) \right| = 0  \textnormal{ for any }  x \in \mathcal{S}.
\end{align}
Let $(X(t))_{ t \geq 0}$ be the Markov process described before. Relation \eqref{ergodicconvergence} implies that for any $A \subset \mathcal{S}$, the probability of the event $\{ X(t) \in A \}$ converges to $\pi(A)$ as $t \to \infty$, irrespective of the initial state $X(0)$. 
This is same as saying that the reaction dynamics $(X(t))_{ t \geq 0}$ is \emph{ergodic} with $\pi$ as the unique stationary distribution. Ergodicity implies that for any real-valued function $f$ satisfying $\sum_{y \in  \mathcal{S} } |f(y)| \pi(y) < \infty$, we have
\begin{align}
\label{ergodicconvergenceoff}
\lim_{t \to \infty} \E( f(X(t)) ) &= \sum_{y \in  \mathcal{S} } f(y) \pi(y)
 \end{align}
 Moreover the following limit holds with probability $1$
 \begin{align}
\textnormal{ and } \label{resultLLN}
\lim_{t \to \infty} \frac{1}{t} \int_{0}^{t} f(X(s))ds &= \sum_{y \in \mathcal{S}} f(y) \pi(y).
\end{align}
 Relation \eqref{ergodicconvergenceoff} can be used to show that the moments of the reaction dynamics converge to their \emph{steady state} values as $t \to \infty$. Relation \eqref{resultLLN} is just the ergodic theorem for Markov processes (see \cite{Norris}) and it shows that the stationary distribution of the population can be inferred by observing a single trajectory of the underlying Markov process $(X(t))_{ t \geq 0}$ for a sufficiently long time. 

Recall the definition of $p_x(t,y)$ from \eqref{transition_probabilities}.  We say that a state $y \in \mathcal{S}$ is \emph{accessible} from another state $x\in \mathcal{S}$ if $p_x(t,y) > 0$ for some $t > 0$. For the reaction dynamics to be ergodic it is necessary that the state space $\mathcal{S}$ is \emph{irreducible}, which means that all the states in $\mathcal{S}$ are \emph{accessible} from each other. Assuming irreducibility, it follows from the work of \cite{Meyn1}, that ergodicity can be checked by showing the existence of a positive function $V $ on $\mathcal{S}$ such that $V(x) \to \infty$ as $\|x\| \to \infty$ and for some $c_1,c_2 >0$, the following holds for all $x \in \mathcal{S}$:
\begin{align}
\label{negativedriftcondition}
\sum_{k=1}^K \lambda_k(x) \left(  V(x +\nu'_k - \nu_k) -  V(x) \right) \leq c_1 - c_2 V(x).
\end{align}
 In fact, if such a function $V$ exists then the convergence in \eqref{ergodicconvergence} is exponentially fast. 
The left side of \eqref{negativedriftcondition} is the drift the process $( V(X(t)) )_{t \geq 0}$ experiences when $X(t) =x$. Relation \eqref{negativedriftcondition} implies that this drift is negative for large values of $\|x\|$. From now on, we will refer to \eqref{negativedriftcondition} as the \emph{negative drift} condition.

\section{Main Results} \label{sec:mainresults}

In this section we present our framework for checking ergodicity in stochastic reaction networks. Our first task is to select the \emph{right} state space $\mathcal{S}$, so as to ensure that it is irreducible under the reaction dynamics. The most common choice of $\mathcal{S}$ is $\N^d_0$, which corresponds to the situation where each species can have \emph{any} number of molecules with a positive probability. Of course, this will not be true if certain species satisfy a \emph{conservation} relationship which is preserved by all the reactions. For example, in some gene-expression networks (see \cite{Vilar}), the active and inactive states of genes are represented as different species, and hence their will be conserved throughout the dynamics. When conservation relationships are present between $d_c (< d)$ species, then by renaming species if necessary, one can often show that the state space of the form $\mathcal{S} = \N_0^{d-d_c} \times \mathcal{E}_c$ is irreducible, where $\mathcal{E}_c$ is a finite subset of $\N_0^{d_c}$.

Using some recent results from \cite{Craciun}, we show how irreducibility of $\mathcal{S}$ can be checked in Section \ref{sec:irred}. For convenience, we separate the two cases mentioned above, $\mathcal{S} =\N^d_0$ and $\mathcal{S} = \N_0^{d-d_c} \times \mathcal{E}_c$. Once irreducibility is established, ergodicity can be verified by checking a negative-drift condition of the form \eqref{negativedriftcondition}. This is done in Section \ref{sec:negdrift} using ideas that are developed in significantly greater detail in \cite{Ourpaper}.

\subsection{Checking irreducibility} \label{sec:irred}  

For the reaction network described in Section \ref{sec:prelim}, we define its \emph{structure} to be the set $\mathcal{R} = \{  ( \nu_k , \nu'_k ) : k =1,\dots,K\}$. This structure along with the vector of positive rate constants $\theta = (\theta_1,\dots,\theta_K)$ fully determine the stochastic reaction dynamics. Irreducibility is a \emph{structural} property in the sense that it only depends on the network structure ($\mathcal{R}$) and not on the rate constants $(\theta)$. To see this, define a relation between the states in $\mathcal{S}$ as follows : $x \stackrel{ \mathcal{R}  }{\longrightarrow} y $ if and only if $x \geq \nu_k $ and $ y = x+\nu'_k -\nu_k$ for some $k=1,\dots,K$. Let $\stackrel{ \mathcal{R}^{*}  }{\longrightarrow}$ be the \emph{transitive closure} of this relation. In other words, $x \stackrel{ \mathcal{R}^{*} }{\longrightarrow} y$ if and only if there exist states $z_1,z_2,\dots,z_{n-1}$ for some $n\geq 1$ such that 
\begin{align}
\label{chainrelationship}
x = z_0 \stackrel{ \mathcal{R}  }{\longrightarrow} z_1  \stackrel{ \mathcal{R}  }{\longrightarrow} z_2 \dots \stackrel{ \mathcal{R}  }{\longrightarrow} z_n \stackrel{ \mathcal{R}  }{\longrightarrow} z_n = y. 
\end{align}
For each $k =1,\dots,K$ let $n_k$ be the number of elements in the set $\{ i = 1,\dots,n :  z_i = z_{i-1} + \nu'_k -\nu_k \}$. Then $\sum_{k=1}^K n_k = n$ and
\begin{align}
\label{spanrelationship}
y =  x + \sum_{k=1}^K (\nu'_k - \nu_k) n_k.
\end{align}
Observe that the form of the function $\lambda_k$ (see \eqref{massactionkinetics}) implies that $\lambda_k(z) > 0$ is equivalent to the condition $z \geq \nu_k$. This shows that when the state is $z$, the reaction $k$ has a positive probability of firing if and only if $z \geq \nu_k$. Hence $p_{x}(t,y) >0$ for some $t>0$, if and only if $x \stackrel{ \mathcal{R}^{*}  }{\longrightarrow} y$. This proves our claim that irreducibility is a structural property.

Let $M$ be the $d \times K$ matrix whose $k$-th column is $(\nu'_k - \nu_k)$. Then $M$ is the \emph{stoichiometry matrix} for the reaction network with structure $\mathcal{R}$. Suppose there is a non-zero vector $\gamma \in \R^d_+$, in the left null-space of $M$ 
\begin{align}
\label{leftnullspacem}
\gamma^T M = \bar{0}_K.
\end{align}
In this case, $ \langle \gamma, X(t)  \rangle =  \langle \gamma, X(0)  \rangle$ for all $t \geq 0$, where $(    X(t)  )_{t \geq 0}$ is the Markov process representing the reaction dynamics. Therefore the species in the set $\{ \mathbf{S}_i  :  i \in \textnormal{supp}(\gamma)\}$ satisfy a conservation relation and $\mathcal{S} = \N_0^{d}$ cannot be irreducible.
Of course a non-zero $\gamma$ satisfying \eqref{leftnullspacem} cannot be present if $\textnormal{Rank}(M) = d$, and in this case we can expect $\mathcal{S} = \N_0^{d}$ to be irreducible. We consider this situation first and deal with the other situation later.

{\bf Networks with no conservation relations} : We now present sufficient conditions to check if $\mathcal{S} = \N^d_0$ is irreducible for the reaction network. We need to verify that for every $x,y \in \N^d_0$ we have $x \stackrel{ \mathcal{R}^{*} }{\longrightarrow} y$. From \eqref{spanrelationship} it is immediate that this can only be true if
\begin{align}
\label{colspancond1}
\textnormal{Colspan}_{\N_0}( M ) = \Z^d.
\end{align}
Checking \eqref{colspancond1} directly is computationally difficult. However \eqref{colspancond1} is equivalent to having $\textnormal{Colspan}_{\Z}( M )  = \Z^d $ and $ \textnormal{Colspan}_{\R_+}( M )  = \R^d$ (see Theorem 3.4 in \cite{Craciun}).
The first condition, $\textnormal{Colspan}_{\Z}( M )  = \Z^d$, can be easily checked by computing the \emph{Hermite normal form} (see \cite{Cohen}) of the matrix $M$. The Hermite normal form is an analogue of the row-reduced echelon form for integer matrices. Assuming $\textnormal{Rank}(M) = d$, it follows from Theorem 3.6 in \cite{Davis} that the second condition, $ \textnormal{Colspan}_{\R_+}( M )  = \R^d$, can be checked by showing that there exists a vector $v \in \R^K$ with strictly positive entries, satisfying $M v = \bar{0}_d$. Such a vector exists if and only if the LFP corresponding to 
\begin{align}
\label{lfp1}
\mathcal{F}_1= \left\{ v \in \R^{K} : M v = \bar{0}_d \textnormal{ and } v \geq \bar{1}_K \right\}
\end{align}
has a solution. Note that this LFP has the form \eqref{lfp:form} with $A= -I_K$, $b = -\bar{1}_K$, $A_{\textnormal{eq}} = M$ and $b_{\textnormal{eq}} = \bar{0}_d$.

Assuming \eqref{colspancond1}, the analysis in \cite{Craciun} shows that for some large positive vector $z_0 \in \N^d$ all the states in the region $\{z \in \N^d_0 : z \geq z_0\}$ are accessible from each other. Moreover to prove that $\N^d_0$ is irreducible we only have demonstrate that for some $x,x' \in \N^d$:
\begin{align}
\label{suff:condirred}
 \bar{0}_d \stackrel{ \mathcal{R}^{*}  }{\longrightarrow} x  \quad \textnormal{ and } \quad  x' \stackrel{ \mathcal{R}^{*}  }{\longrightarrow}  \bar{0}_d.
\end{align}
 For details, see Theorem 3.8 in \cite{Craciun}.

From now on let $\mathcal{D} =\{1,\dots,d\}$ be the set of species and $\mathcal{K} = \{1,\dots,K\}$ be the set of reactions. To prove the first accessibility relation in \eqref{suff:condirred}, we need to show that there exists a sequence of $n$ reactions $k_1,\dots, k_n \in \mathcal{K}$, such that the cumulative effect of all these $n$ reactions is positive for each species (that is, $\sum_{j=1}^n (\nu'_{k_j} - \nu_{k_j} ) > \bar{0}_d$) and each intermediate reaction $k_j$ has a positive probability of firing (that is, $\sum_{l=1}^{j-1} (\nu'_{k_l} - \nu_{k_l} ) \geq \nu_{k_j}$). Such a sequence of reactions is difficult to construct for general networks, but we now present a simple scheme that allows us to easily check if such a sequence exists for a large class of networks.

Our scheme is motivated by the observation that many biochemical reaction networks appear like \emph{complex cascades of birth-death networks}. By this we mean that in these networks, a certain set of species are produced \emph{constitutively} due to reactions of the form $\emptyset \longrightarrow {\bf S}_i$. These species then produce another set of species which in turn produce another set of species and so on. If all the species are produced this way then one can construct a sequence of reactions that proves the first accessibility relation in \eqref{suff:condirred}. To make this formal, we arrange the species into \emph{levels} according to the minimum number of reactions it takes for the species to be produced from nothing. 
Let $H_0 = \emptyset$ and for each $l=1,2,\dots$ define
\begin{align*}
G_l =& \{  i \in \mathcal{D} \backslash H_{l-1} :   \textnormal{supp}(\nu_k) \subset H_{l-1} \textnormal{ and } i \in \textnormal{supp}(\nu'_k)    \textnormal{ for some } k \in \mathcal{K} \} \quad  \textnormal{ and } \quad H_l = H_{l-1} \cup G_l,
\end{align*}
where $\mathcal{D} \backslash H_{l-1}  = \{ i \in \mathcal{D}  : i \notin   H_{l-1} \}$. The set $G_l$ contains all the species at level $l$ and the set $H_l$ contains all the species that belong to levels $1,\dots,l$. We say that a reaction network with structure $\mathcal{R}$ is \emph{exhaustive} if there exists a $l_0 \geq 1$ such that $H_{l_0} = \cup_{l=1}^{l_0} G_l = \mathcal{D} $.
The level construction allows us to prove the following.
\begin{lemma}
\label{lem:level} 
Suppose that a reaction network with structure $\mathcal{R}$ is exhaustive. Then there exists a $x \in \N^d$ such that $\bar{0}_d \stackrel{ \mathcal{R}^{*}  }{\longrightarrow} x$.
\end{lemma}
\begin{proof}
We prove this lemma by an induction argument. In this proof we denote the relation $\stackrel{ \mathcal{R}^{*}  }{\longrightarrow} $ by $\longrightarrow$. We say that a level $l$ is \emph{satisfiable} if for any $r \in \N^d_0$ with $\textnormal{supp}(r) \subset H_l$, we can find a state $x$ such that $x \geq r$ and $\bar{0}_d \longrightarrow x$. Certainly level $1$ is satisfiable, because $H_1 = G_1$ consists of those species that are produced from nothing. Suppose that level $(l-1)$ is satisfiable. Pick any $r \in \N^d_0$ with $\textnormal{supp}(r) \subset H_l$. We can write it as $r = r_1+r_2$ where $\textnormal{supp}(r_1) \subset H_{l-1}$ and $\textnormal{supp}(r_2) \subset G_l$. Note that molecules of species in $G_l$ are produced by consuming molecules of species in $H_{l-1}$. Hence we can find states $x,y$ with $x\geq r$, $y \geq r_1$ and $\textnormal{supp}(y) \subset H_{l-1}$ such that $ y \longrightarrow x$. Satisfiability of level $l-1$ implies that there exists a $\delta \in \N^d_0$ such that $\textnormal{supp}(\delta) \subset H_{l-1}$ and $\bar{0}_d \longrightarrow y + \delta$. But $ y \longrightarrow x$ implies that $ y +\delta \longrightarrow x +\delta$ and hence $\bar{0}_d \longrightarrow x + \delta$. This shows that level $l$ is satisfiable and by induction we can conclude that all the levels are satisfiable. Since $\mathcal{R}$ is exhaustive we can find a state $x$ with $\textnormal{supp}(x)  = \mathcal{D}$ such that $\bar{0}_d \longrightarrow x$. This completes the proof of the lemma.
\end{proof}

Using Lemma \ref{lem:level} we can check the first relation in \eqref{suff:condirred}. To check the second relation we consider a reaction network with the \emph{inverse} structure $\mathcal{R}_{ \textnormal{inv} } = \{  ( \nu'_k , \nu_k ) : k =1,\dots,K\} $, which is obtained by \emph{flipping} the arrows in \eqref{formofthereaction}. We can define the relation $\stackrel{ \mathcal{R}^{*}_{\textnormal{inv}}  }{\longrightarrow} $ for this network structure as above. Note that for any $x,y \in \N^d_0$, $x \stackrel{ \mathcal{R}^{*}_{\textnormal{inv}}  }{\longrightarrow} y$ holds if and only if $y \stackrel{ \mathcal{R}^{*}  }{\longrightarrow} x$ holds. Hence the second relation in \eqref{suff:condirred} can be checked by showing that $\bar{0}_d \stackrel{ \mathcal{R}^{*}_{\textnormal{inv}}  }{\longrightarrow} x' $ for some $x' \in \N^d$. This can again be done using Lemma \ref{lem:level} if the network with structure $ \mathcal{R}_{\textnormal{inv}}$ is exhaustive.

The above discussion gives us our main result for checking the irreducibility of $\mathcal{S} = \N^d_0$.
\begin{theorem}
\label{thm:irrednocons}
Suppose $\textnormal{Rank}(M) = d$, $\textnormal{Colspan}_{\Z}( M )  = \Z^d$ and the LFP corresponding to $\mathcal{F}_1$ (see \eqref{lfp1}) has a solution. Also assume that reaction networks with structures $\mathcal{R}$ and $\mathcal{R}_{ \textnormal{inv} }$ are exhaustive. Then the state space $\mathcal{S} = \N^d_0$ is irreducible under the reaction dynamics.
\end{theorem}

{\bf Networks with conservation relations} : We now come to the situation where the reaction network has conservation relations. Each conservation relation corresponds to a non-zero vector $\gamma \in \R^d_+$ satisfying \eqref{leftnullspacem}. We assume that the network has only one conservation relation, but our method can be easily extended to cases where many conservation relations are present.

Let $\gamma$ be as above and suppose that $\textnormal{supp}(\gamma)$ contains $d_c$ elements, where $d_c<d$. Then the reaction network has $d_c$ conserved species, while the remaining $d_u = d - d_c$ species are \emph{unconserved}. By renaming species if necessary, we can assume that $\gamma = (0,\dots,0,\gamma_{d_u+1},\dots,\gamma_d )$, and hence the sets of conserved and unconserved species are given by $\mathcal{D}_c = \{d_u+1,\dots,d\}$ and $\mathcal{D}_u = \{1,\dots,d_u\}$ respectively. Let $\mathcal{E}_c$ be the finite subset of $\N^{d_c}_0 $ defined by
\begin{align*}
\mathcal{E}_c = \left\{ ( x_1,\dots, x_{d_c} ) \in \N^{d_c}_0  : \sum_{i=1}^{d_c} \gamma_{d_u+i} x_i = C \right\},
\end{align*}
where $C$ is some constant. For each $k \in \mathcal{K}$, let $\bar{\nu}_k \in \N^{d_u}_0$ and $\hat{\nu}_k \in \N^{d_c}_0$ be the vectors containing the first $d_u$ and the last $d_c$ components of $\nu_k$. The definitions of $\bar{\nu}'_k$ and $\hat{\nu}'_k$ are similar. Define $\bar{M}$ to be the $d_u \times K$ matrix whose $k$-th column is $(\bar{\nu}'_k - \bar{\nu}_k)$.

We now describe a way to show that state space $\mathcal{S} = \N_0^{d_u} \times \mathcal{E}_c $ is irreducible for the reaction dynamics. For this to hold it is necessary that
\begin{align}
\label{colspancond2}
\textnormal{Colspan}_{\N_0}( \bar{M} ) = \Z^{d_u}.
\end{align}
This condition can be checked by verifying that $\textnormal{Rank}(\bar{M}) =d_u$, $\textnormal{Colspan}_{\Z}( \bar{M} ) = \Z^{d_u}$ and the LFP corresponding to 
\begin{align}
\label{lfp2}
\mathcal{F}_2= \left\{ v \in \R^{K} : \bar{M} v = \bar{0}_{d_u} \textnormal{ and } v \geq \bar{1}_K \right\}
\end{align}
has a solution. Assuming \eqref{colspancond2}, the irreducibility of $\N_0^{d_u} \times \mathcal{E}_c$ can be proved by arranging the unconserved species into levels as before. However the description of levels gets more complicated because of the presence of conserved species.

For any group of unconserved species $A \subset \mathcal{D}_u$ and any $e \in \mathcal{E}_c$ define $\mathcal{K}(A,e) = \{ k \in \mathcal{K}  :  \textnormal{supp}( \bar{\nu}_k) \subset A \textnormal{ and } e \geq \hat{\nu}_{k} \}$. This is the set of reactions which have a positive probability of firing when the molecules of species in $A$ are abundantly available, and when the dynamics of the conserved species is at state $e$. 
Suppose that the finite set $\mathcal{E}_c$ has $n_c$ elements. Then we can write it as $\mathcal{E}_c = \{e_1,\dots,e_{n_c}\}$. For any $A \subset \mathcal{D}_u $ we define a $n_c \times n_c$ matrix $Z(A)$ as 
\begin{align*}
Z_{ij}(A) = \left\{
\begin{tabular}{cl}
$1$ & if $e_j = e_i + \hat{\nu}'_k-\hat{\nu}_k$ \ for some $k \in \mathcal{K}(A,e_i)$ \\  
$0$ & otherwise.
\end{tabular} \right.
\end{align*}
Note that $Z_{ij}(A) =1$ if and only if the dynamics of the conserved species can reach $e_j$ from $e_i$ due to the firing of a single reaction in $\mathcal{K}(A,e_i)$. If we define $$\Omega(A) = (I_{n_c} + Z(A) )^{n_c-1} ,$$ then $\Omega_{ij}(A) > 0$ if and only if there exist $i_1,\dots,i_n \in \{1,\dots,n_c\}$ such that
\begin{align*}
Z_{i i_1 }(A) = Z_{i_1 i_2}(A) = \dots =  Z_{i_{n-1} i_n}(A) = Z_{i_{n} j}(A) = 1.
\end{align*}
We can define a relation on $\mathcal{E}_c$ as follows : $e_i \leftrightarrow_A e_j$ if and only if $\Omega_{ij}(A) = \Omega_{ji}(A) = 1$. This is an equivalence relation and hence we can partition $\mathcal{E}_c$ into $\eta(A)$ equivalence classes. An equivalence class $C$ is called \emph{closed} if for each $i,j \in \{1,\dots,n_c\}$, if $e_i \in C$ and $Z_{ij}(A) = 1$ then $e_j \in C$. An equivalence class is called \emph{open} if it is not closed.
Let $\mathcal{C}(A)$ be the collection of all closed equivalence classes corresponding to the relation $\leftrightarrow_A$. Each closed equivalence class consists of states of the conserved species that are accessible from each other given that the molecules of the unconserved species in $A$ are abundantly available. For the state space $\mathcal{S} = \N_0^{d_u} \times \mathcal{E}_c$ to be irreducible for the reaction dynamics, it is necessary that when all the unconserved species are abundantly available ($A = \mathcal{D}_u$), then the relation $\leftrightarrow_A$ induces only one closed equivalence class that covers the whole set $\mathcal{E}_c$. This necessary condition can be stated as $\eta(\mathcal{D}_u) = 1$.

We are now ready to classify our unconserved species into various levels. Let $H_0 = \emptyset$ and for each $l=1,2,\dots$ define
\begin{align*}
G_l  = & \left\{ i \in \mathcal{D}_u \backslash H_{l-1} : \textnormal{for each } C \in \mathcal{C}(H_{l-1}) \textnormal{ there exists a }  k \in \mathcal{K}(H_{l-1},C) \textnormal{ such that } i \in \textnormal{supp}(\nu'_k) \right\} 
\end{align*}
and $H_l = H_{l-1} \cup G_l$, where $\mathcal{K}(H_{l-1},C) = \cup_{e \in C}\mathcal{K}(H_{l-1},e)$. We say that a reaction network with structure $\mathcal{R}$ is \emph{exhaustive} if there exists a $l_0 \geq 1$ such that $H_{l_0} = \mathcal{D}_u$. Analogous to Lemma \ref{lem:level} we get the following result.
\begin{lemma}
\label{lem:level2} 
Suppose that a reaction network with structure $\mathcal{R}$ is exhaustive and $\eta(\mathcal{D}_u) = 1$. Then there exists a $x \in \N^{d_u}$ such that for any $e,f \in \mathcal{E}_c$ we have
\begin{align*}
 (\bar{0}_{d_u}, e) \stackrel{ \mathcal{R}^{*}  }{\longrightarrow} ( x , f).  
 \end{align*}
\end{lemma}
\begin{proof}
Observe that for any $A \subset \mathcal{D}_u$, if the dynamics of the conserved species is at a state which is inside an open equivalence class of $\leftrightarrow_A$, then this dynamics will reach a closed equivalence class after a finite number of transitions. The proof of this lemma is essentially the same as the proof of Lemma \ref{lem:level}. The only difference is that to produce the species in $G_l$ one has to choose reactions based on the current state of the conserved species, which varies due to transitions inside $\mathcal{E}_c$, but eventually gets trapped inside a closed equivalence class of $\leftrightarrow_{H_{l-1}}$.
\end{proof}

Defining $ \mathcal{R}_{\textnormal{inv}}$ as before, we get our main result for checking the irreducibility of 
$\mathcal{S} = \N^{d_u}_0 \times \mathcal{E}_c$.  
\begin{theorem}
\label{thm:irredconsrels}
Suppose $\textnormal{Rank}(\bar{M}) = d_u$, $\textnormal{Colspan}_{\Z}( \bar{M} )  = \Z^{d_u}$, $\eta(\mathcal{D}_u) = 1$ and the LFP corresponding to $\mathcal{F}_2$ (see \eqref{lfp2}) has a solution. Also assume that reaction networks with structures $\mathcal{R}$ and $\mathcal{R}_{ \textnormal{inv} }$ are exhaustive. Then the state space $\mathcal{S} = \N^{d_u}_0 \times \mathcal{E}_c$ is irreducible under the reaction dynamics.
\end{theorem}

\subsection{Checking the negative drift condition} \label{sec:negdrift}  
Suppose that the state space of the form $\mathcal{S} = \N^{d_u}_0 \times \mathcal{E}_c$ has been shown to be irreducible under the reaction dynamics, where the set $\mathcal{E}_c$ may be empty. This covers both the situations discussed in Section \ref{sec:irred}. We also assume that $\sum_{i =1}^d \nu_{ik} \leq 2$ for each $k \in \mathcal{K}$. This implies that all the reactions are either \emph{constitutive} ($\emptyset \longrightarrow \star $), \emph{unary} ($S_i \longrightarrow \star $) or \emph{binary} ($S_i + S_j \longrightarrow \star $).  

Define a set of reactions by $\mathcal{K}_{\textnormal{unr}} = \{ k \in \mathcal{K} : \sum_{i =1}^d \nu_{ik} =1 \textnormal{ and } \textnormal{supp}(\nu_{k} ) \subset \mathcal{D}_u \} .$ Each reaction $k \in \mathcal{K}_{\textnormal{unr}}$ has the form $S_i \longrightarrow \star $ for some unconserved species $S_i$. For such a $k$ define a vector $a_k = (0,\dots,0,1,0,\dots,0) \in \N^{d_u}_0$, where the $1$ is at the $i$-th place.
 Let $\mathcal{K}_{\textnormal{bin}}$ be the set of all binary reactions $\mathcal{K}_{\textnormal{bin}} = \{ k \in \mathcal{K} : \sum_{i =1}^d \nu_{ik} =2 \} $
and let $K_q$ be the number of reactions in $\mathcal{K}_{\textnormal{bin}} $.

Recall that $\theta_k$ is the rate constant for the $k$-th reaction.
Define a $d_u \times d$ matrix by
\begin{align*}
A = \sum_{k \in \mathcal{K}_{\textnormal{unr}} } \theta_k a_k ( \nu'_k - \nu_k )^T.
\end{align*}
 Let $M_{q}$ be the $d \times K_q$ matrix whose set of  columns is $\{ (\nu'_k - \nu_k) : k \in \mathcal{K}_{\textnormal{bin}}   \}$. The next lemma will help us check \eqref{negativedriftcondition}. The conditions in this lemma are taken from Proposition 10 in \cite{Ourpaper}. 
\begin{lemma}\label{lem:negdrift}
Suppose there exists a vector $w \in \R^d$ whose first $d_u$ components are strictly positive, and $w$ satisfies $Aw < \bar{0}_{d_u}$ and $w^T M_q  = \bar{0}^T_{K_q}$. Then there exists a positive function $V$ on $\mathcal{S}$ along with constants $c_1,c_2 >0$ such that $V(x) \to \infty$ as $\|x\| \to \infty$, and \eqref{negativedriftcondition} is satisfied for all $x  \in \mathcal{S} =\N^{d_u}_0 \times \mathcal{E}_c  $.
\end{lemma}
\begin{proof}
Let $\gamma \in \R^d_+$ be the vector that characterizes the conservation relation in the network. The last $d_c$ components of $\gamma$ are strictly positive, and since $\gamma$ satisfies \eqref{leftnullspacem} we have $A \gamma = \bar{0}_{d_u}$ and $\gamma^T M_q = \bar{0}^T_{K_q}$. Therefore we can choose an $\alpha >0$ such that the vector $v = w + \alpha \gamma$ has all strictly positive components and $v$ satisfies $Av < \bar{0}_{d_u}$ and $v^T M_q  = \bar{0}^T_{K_q}$. Define the function $V : \mathcal{S} \to (0,\infty)$ by $V(x) = v^T x$. 
The relation $v^T M_q = \bar{0}^T_{K_q}$ implies that for any $k \in \mathcal{K}_{\textnormal{bin}} $, we have 
$V(x +\nu'_k - \nu_k) -  V(x) = (\nu'_k - \nu_k)^T v = 0.$
For any $k \in \mathcal{K}_{\textnormal{unr}}$, $\lambda_k(x) = \theta_k x^T a_k$ and for any $k \in \mathcal{K}' = \{k \in \mathcal{K} : k \notin  \mathcal{K}_{\textnormal{unr}} \cup \mathcal{K}_{\textnormal{bin}}  \}$, the function $x \mapsto \lambda_k(x)$ is bounded on $\mathcal{S} $. Therefore there exists a $c_1 >0$ such that for all $x  = (x_1,x_2)=\N^{d_u}_0 \times \mathcal{E}_c $, the left side of \eqref{negativedriftcondition} is less than
\begin{align*}
 c_1 +  \sum_{k \in \mathcal{K}_{\textnormal{unr}} } \theta_k x_1^T  a_k ( \nu'_k - \nu_k )^Tv = c_1 +  x_1^T Av.
\end{align*}
Since $Av < \bar{0}_{d_u}$ and $\mathcal{E}_c$ is finite, we can find a $c_2 >0$ such that \eqref{negativedriftcondition} is satisfied for all $x  \in \mathcal{S}$. This completes the proof of the lemma.
\end{proof}

Define a $(2d_u) \times d$ matrix by $B = [-I_{d_u}  \ N  ]$, where $N$ is the $d_u \times d_c$ matrix of zeroes. 
Observe that a vector $w$ satisfying the conditions of Lemma \ref{lem:negdrift} exists if and only if the LFP corresponding to the set
\begin{align*}
\mathcal{F}_3 = \left\{ v \in \R^d :  \left[ \begin{array}{c} A \\B  \end{array} \right] v \leq - \left[ \begin{array}{c} \bar{1}_{d_u} \\ \bar{1}_{d_u}  \end{array} \right] \textnormal{ and } M^T_q v = \bar{0}_{K_q} \right\} 
\end{align*}
has a solution. This solution, Lemma \ref{lem:negdrift} and Theorem 7.1 in \cite{Meyn1} prove the ergodicity of the reaction dynamics, giving us our last result.
\begin{theorem}
\label{thm:ergodic}
Assume that the state space $\mathcal{S} = \N^{d_u}_0 \times \mathcal{E}_c$ is irreducible for the reaction dynamics and the LFP corresponding to $\mathcal{F}_3$ has a solution. Then the relation \eqref{ergodicconvergence} holds and the stochastic reaction dynamics is ergodic.
\end{theorem}

\section{An Example} \label{sec:example}
To illustrate our procedure for checking ergodicity, we consider the example of the genetic oscillator described in \cite{Vilar}. It has $9$ species $\mathbf{S_1},\dots,\mathbf{S_9}$ and 16 reactions given in the table below. 
\begin{table}[h]
\label{ta:gen}
\begin{center}
\caption{List of reactions for genetic oscillator}\label{reactionnetwork}
\begin{tabular}{cl|cl}
No. & Reaction & No. & Reaction \\\hline
1 & $\mathbf{S_6} + \mathbf{S_2} \longrightarrow \mathbf{S}_7 $ & 9 & $\mathbf{S_2} \longrightarrow \emptyset $ \\
2 & $\mathbf{S}_7 \longrightarrow  \mathbf{S_6} + \mathbf{S_2} $ & 10 & $\mathbf{S_9} \longrightarrow \mathbf{S}_9 + \mathbf{S_3}  $ \\
3 & $\mathbf{S_8} + \mathbf{S_2} \longrightarrow \mathbf{S}_9 $ & 11 & $\mathbf{S_8} \longrightarrow \mathbf{S}_8 + \mathbf{S_3}  $ \\
4 & $\mathbf{S}_9 \longrightarrow   \mathbf{S_8} + \mathbf{S_2} $ & 12 & $ \mathbf{S_3} \longrightarrow \emptyset $ \\
5 & $\mathbf{S_7}  \longrightarrow \mathbf{S}_7+ \mathbf{S_1} $ & 13 & $\mathbf{S_3} \longrightarrow \mathbf{S}_3 + \mathbf{S_4}  $ \\
6 & $\mathbf{S_6}  \longrightarrow \mathbf{S}_6+ \mathbf{S_1}  $ & 14& $ \mathbf{S_4} \longrightarrow \emptyset $ \\
7 & $ \mathbf{S_1} \longrightarrow \emptyset$ & 15 & $\mathbf{S_2} + \mathbf{S_4} \longrightarrow \mathbf{S}_5$ \\
8 & $\mathbf{S_1} \longrightarrow \mathbf{S_1} + \mathbf{S}_2 $ & 16 & $\mathbf{S_5} \longrightarrow \mathbf{S}_4 $ \\
\hline
\end{tabular}
\end{center}
\end{table}
This network has an \emph{activator} gene, which may exist in \emph{bound} ($\mathbf{S}_6$) or \emph{unbound} ($\mathbf{S}_7$) form. Similarly there is a \emph{promoter} gene which may also exist in \emph{bound} ($\mathbf{S}_8$) or \emph{unbound} form ($\mathbf{S}_9$). We assume that one copy of both the genes is present. Hence the sum of the species numbers of $\mathbf{S}_6$ and $\mathbf{S}_7$ is $1$. The same is true for species $\mathbf{S}_8$ and $\mathbf{S}_9$. Note that we have named the species in the model of \cite{Vilar} in such a way, so that the conserved species are at the end. Even though our procedure will work for any choice of rate constants ($\theta_k$), we set all of them to $1$ for convenience.

For this network, the set of unconserved species is $\mathcal{D}_u = \{1,2,3,4,5\}$, and there are two disjoint sets of conserved species $\mathcal{D}^{(1)}_{c} = \{6,7\}$ and $\mathcal{D}^{(2)}_{c} =\{8,9\}$. The dynamics of both sets of conserved species is over the set $\mathcal{E} = \{(0,1) , (1,0)\}$. To prove ergodicity we first need to show that the state space $\mathcal{S} = \N^5_0 \times \mathcal{E} \times \mathcal{E}$ is irreducible. For this we use Theorem \ref{thm:irredconsrels}, generalized to the case of having two disjoint sets of conserved species.

Consider the dynamics of species in $\mathcal{D}^{(1)}_{c} = \{6,7\}$ over $\mathcal{E} $. For any $A \subset \mathcal{D}_u$, the relation $\leftrightarrow_A$  induces only one closed equivalence class. This class is either $\mathcal{E}$ or $\{(1,0)\}$  depending on whether $2 \in A$ or not. By symmetry one can see that exactly the same holds true for the dynamics of species in $\mathcal{D}^{(2)}_{c} = \{8,9\}$. With this information we can arrange the unconserved species into levels as : $G_1 = \{1,3\}$, $G_2 =\{2,4\}$ and $G_3 = \{5\}$, which shows that the network is exhaustive. Similarly for the inverse network we can arrange the unconserved species into levels as : $G_1 = \{1,2,3,4\}$ and $G_2 =\{5\}$, thereby showing that the inverse network is also exhaustive. Other conditions of Theorem \ref{thm:irredconsrels} can be easily checked and hence this result proves that the state space $\mathcal{S} = \N^5_0 \times \mathcal{E} \times \mathcal{E}$ is irreducible.

Now we need to check the negative drift condition. Observe that $\mathcal{K}_{ \textnormal{unr}} = \{7,8,9,12,13,14,16\}$ and $\mathcal{K}_{ \textnormal{bin}} = \{1,3,15\}$. Constructing matrices $A,B$ and $M_q$ from Section \ref{sec:negdrift}, one can verify that the vector $$v= (2,1,2,1,2,-0.5,0.5,-0.5,0.5)$$ solves the feasibility problem for $\mathcal{F}_3$. Theorem \ref{thm:ergodic} proves that the reaction dynamics is ergodic with state space $\mathcal{S} = \N^5_0 \times \mathcal{E} \times \mathcal{E}$.




\begin{thebibliography}{1}

\bibitem{Ourpaper}
C.~Briat, A.~Gupta, and M.~Khammash.
\newblock A scalable computational framework for establishing long-term
  behavior of stochastic reaction networks.
\newblock {\em Submitted. Available on arXiv:1304.5404}, 2013.

\bibitem{Cohen}
H.~Cohen.
\newblock {\em A course in computational algebraic number theory}, volume 138
  of {\em Graduate Texts in Mathematics}.
\newblock Springer-Verlag, Berlin, 1993.

\bibitem{Davis}
C.~Davis.
\newblock Theory of positive linear dependence.
\newblock {\em Amer. J. Math.}, 76:733--746, 1954.

\bibitem{Elowitz}
M.~B. Elowitz, A.~J. Levine, E.~D. Siggia, and P.~S. Swain.
\newblock Stochastic gene expression in a single cell.
\newblock {\em Science}, 297(5584):1183--1186, 2002.

\bibitem{GP}
D.~T. Gillespie.
\newblock Exact stochastic simulation of coupled chemical reactions.
\newblock {\em The Journal of Physical Chemistry}, 81(25):2340--2361, 1977.

\bibitem{Meyn1}
S.~P. Meyn and R.~L. Tweedie.
\newblock Stability of {M}arkovian processes. {III}. {F}oster-{L}yapunov
  criteria for continuous-time processes.
\newblock {\em Adv. in Appl. Probab.}, 25(3):518--548, 1993.

\bibitem{Norris}
J.~R. Norris.
\newblock {\em Markov chains}, volume~2 of {\em Cambridge Series in Statistical
  and Probabilistic Mathematics}.
\newblock Cambridge University Press, Cambridge, 1998.
\newblock Reprint of 1997 original.

\bibitem{Craciun}
L.~Pauleve, G.~Craciun, and H.~Koeppl.
\newblock Dynamical properties of discrete reaction networks.
\newblock {\em Available on arXiv:1302.3363}, 2013.

\bibitem{Vilar}
J.~M.~G. Vilar, H.~Y. Kueh, N.~Barkai, and S.~Leibler.
\newblock {Mechanisms of noise-resistance in genetic oscillator}.
\newblock {\em Proc. Natl. Acad. Sci.}, 99(9):5988--5992, 2002.

\end{thebibliography}
\end{document}